\documentclass{amsart}

\usepackage{amsfonts,amssymb,amsmath,amscd,latexsym,amsthm,amsbsy,dsfont,mathtools,hyperref}
\usepackage{cleveref,tikz}
\usetikzlibrary{tikzmark}
\usepackage{appendix}
\usepackage{dsfont}
\usepackage{enumerate,enumitem}
\usepackage{todonotes}

\hypersetup{
    pdftoolbar=true,
    pdfmenubar=true,
    pdffitwindow=false,
    pdfstartview={FitH},
    pdftitle={},
    pdfauthor={},
    pdfsubject={},
    pdfkeywords={}
    pdfnewwindow=true,
    colorlinks=true, 
    linkcolor=blue,
    citecolor=blue,
    urlcolor=black,
}

\newcommand{\R}{\mathds{R}}
\newcommand{\C}{\mathds{C}}

\newcommand{\N}{\mathds{N}}
\newcommand{\Z}{\mathds{Z}}
\renewcommand{\H}{\mathds{H}}
\renewcommand{\S}{\mathds{S}}

\newcommand{\dd}{\mathrm{d}}
\newcommand{\g}{\mathrm{g}}
\newcommand{\gp}{\mathrm{g}_{\mathrm{prod}}}
\newcommand{\gtriv}{\mathrm{g}_{\mathrm{triv}}}
\newcommand{\h}{\mathrm{h}}
\newcommand{\ind}{i_{\mathrm{Morse}}}
\newcommand{\gen}{\operatorname{gen}}

\DeclareMathOperator{\Vol}{Vol}

\usetikzlibrary{positioning}
\newcommand\pgfmathsinandcos[3]{%
  \pgfmathsetmacro#1{sin(#3)}%
  \pgfmathsetmacro#2{cos(#3)}%
}
\newcommand\LongitudePlane[3][current plane]{%
  \pgfmathsinandcos\sinEl\cosEl{#2} 
  \pgfmathsinandcos\sint\cost{#3} 
  \tikzset{#1/.estyle={cm={\cost,\sint*\sinEl,0,\cosEl,(0,0)}}}
}
\newcommand\LatitudePlane[3][current plane]{%
  \pgfmathsinandcos\sinEl\cosEl{#2} 
  \pgfmathsinandcos\sint\cost{#3} 
  \pgfmathsetmacro\yshift{\cosEl*\sint}
  \tikzset{#1/.estyle={cm={\cost,0,0,\cost*\sinEl,(0,\yshift)}}} %
}
\newcommand\DrawLongitudeCircle[2][1]{
  \LongitudePlane{\angEl}{#2}
  \tikzset{current plane/.prefix style={scale=#1}}
  \pgfmathsetmacro\angVis{atan(sin(#2)*cos(\angEl)/sin(\angEl))} %
  \draw[current plane,thin,black] (\angVis:1) arc (\angVis:\angVis+180:1);
  \draw[current plane,thin,dashed] (\angVis-180:1) arc (\angVis-180:\angVis:1);
}

\newcommand\DrawLatitudeCircleblue[2][1]{
  \LatitudePlane{\angEl}{#2}
  \tikzset{current plane/.prefix style={scale=#1}}
  \pgfmathsetmacro\sinVis{sin(#2)/cos(#2)*sin(\angEl)/cos(\angEl)}
  \pgfmathsetmacro\angVis{asin(min(1,max(\sinVis,-1)))}
\draw[current plane,thick,black,blue] (\angVis:1) arc (\angVis:-\angVis-180:1);
\draw[current plane,thick,dashed,blue] (180-\angVis:1) arc (180-\angVis:\angVis:1);
}

\tikzset{%
  >=latex,
  inner sep=0pt,%
  outer sep=2pt,%
  mark coordinate/.style={inner sep=0pt,outer sep=0pt,minimum size=3pt,
    fill=black,circle}%
}

\newtheorem{theorem}{Theorem}
\newtheorem{definition}[theorem]{Definition}
\newtheorem{lemma}[theorem]{Lemma}

\newtheorem{corollary}[theorem]{Corollary}
\newtheorem{proposition}[theorem]{Proposition}

\newtheorem*{sfyp}{\sc Singular fractional Yamabe problem}

\newtheorem{mainthm}{\sc Theorem}

\newtheorem*{mainthm*}{\sc Theorem}

\theoremstyle{remark}
\newtheorem{remark}[theorem]{Remark}

\allowdisplaybreaks
\numberwithin{equation}{section}
\numberwithin{theorem}{section}

\title[Multiplicity in the singular fractional Yamabe problem on spheres]{Multiplicity of singular solutions to the fractional Yamabe problem on spheres}


\author[R. G. Bettiol]{Renato G. Bettiol}
\address{\!\!\!\begin{tabular}{lll}
CUNY Lehman College & & CUNY Graduate Center \\
Department of Mathematics & & Department of Mathematics \\
250 Bedford Park~Blvd W & & 365 Fifth Avenue \\
Bronx, NY, 10468, USA & & New York, NY, 10016, USA
\end{tabular}
}
\email{r.bettiol@lehman.cuny.edu}

\author[M. Gonz\'alez]{Mar\'ia del Mar Gonz\'alez}
\address{
Universidad Autonoma de Madrid \newline \indent Departamento de Matem\'aticas \newline \indent Campus de Cantoblanco and ICMAT \newline \indent 28049 Madrid, Spain}
\email{mariamar.gonzalezn@uam.es}

\author[A. Maalaoui]{Ali Maalaoui}
\address{Clark University \newline \indent Department of Mathematics \newline \indent Worcester, MA, 01610, USA}
\email{amaalaoui@clarku.edu}

\date{\today}

\begin{document}
\begin{abstract}
We prove nonuniqueness results for complete metrics with constant positive fractional curvature conformal to the round metric on $\S^n\setminus \S^k$, using bifurcation techniques. These are singular (positive) solutions to a non-local equation with critical non-linearity.
\end{abstract}

\maketitle

\section{Introduction}

One of the main pillars of modern Conformal Geometry is the Yamabe problem:\-
finding a metric with constant scalar curvature in a given conformal class on a closed manifold $M^n$, $n\geq3$, which corresponds to finding positive solutions to a nonlinear second-order elliptic PDE. A fourth-order version of this problem, for the Paneitz--Branson $Q$-curvature, has been extensively studied, and leads to additional geometric and topological information. Recently, a one-parameter family of curvatures with good conformal properties generalizing the above has been introduced, see e.g.~\cite{Chang-Gonzalez,CC} and the survey \cite{Gonzalez:survey}. This so-called \emph{$\gamma$-fractional curvature}, denoted by $Q_\gamma$ for $\gamma\in(0,\frac{n}{2})$, is a non-local object of order $2\gamma$ defined on $M^n$ when it is the conformal infinity of a conformally compact Einstein manifold, or, more generally, the boundary at infinity of an asymptotically hyperbolic manifold. For particular values of the parameter, $Q_\gamma$ recovers classical curvatures; for instance, $Q_1$ is a multiple of the scalar curvature, $Q_2$ is the usual $Q$-curvature, and $Q_{\frac{1}{2}}$ can be interpreted as the mean curvature of $M$ with respect to its filling.

The fractional Yamabe problem consists of finding a metric in a given conformal class $(M^n,[\g])$ with constant $\gamma$-fractional curvature $Q_\gamma$. We remark that, in the particular case that $\gamma=\frac{1}{2}$, this is known as the Cherrier--Escobar problem for manifolds with boundary, see \cite{Mayer-Ndiaye1} and the references therein.  In general, from the PDE point of view, a positive solution to the fractional Lane--Emden equation
\begin{equation*}
\phantom{\quad\text{on }M,}
P_{\gamma}^{\g}u=Q_{\gamma}^{\g_{u}}\, u^{\frac{n+2\gamma}{n-2\gamma}}\quad\text{on }M
\end{equation*}
yields a metric $\g_u=u^{\frac{4}{n-2\gamma}}\g$ whose $\gamma$-fractional curvature is $Q_{\gamma}^{\g_{u}}$. Here, $P_\gamma^\g$ is the conformal fractional Laplacian operator, a pseudo-differential operator on $M$ of order $2\gamma$ associated to $Q_\gamma$.
In the compact setting, the fractional Yamabe problem can be solved variationally, using an energy functional for $P_\gamma^\g$; this was the initial approach adopted in \cite{MQ}. Nowadays, there are many works exploiting this variational structure \cite{Gonzalez-Wang,Kim-Musso-Wei,Ndiaye-Sire-Sun,Mayer-Ndiaye-DCDS,Mayer-Ndiaye2}, as well as some articles using flows \cite{Jin-Xiong,Daskalopoulos-Sire-Vazquez,Chan-Sire-Sun}.

In this paper, we are interested the following \emph{singular version} of the fractional Yamabe problem, which is much less understood due to the non-local character of the operator; in particular, because it is not well-defined across singularities.

\begin{sfyp}
    Let $(M^n,\g)$ be a closed Riemannian manifold, $\Lambda\subset M$ be a closed subset, and $0<\gamma<\frac{n}{2}$. Find a complete metric on $M\setminus \Lambda$ that is conformal to $\g$ and has constant $\gamma$-fractional curvature.
\end{sfyp}

We restrict our attention to the positive curvature case, which is usually the most involved one. It is known that, on locally conformally flat manifolds, under some additional conditions, $Q_\gamma>0$ implies that the singular set $\Lambda$ has Hausdorff dimension $k<\frac{n}{2}-\gamma$, see \cite{GMS,Zhang}. This is reminiscent of the classical work of Schoen--Yau~\cite{Schoen-Yau} on singular metrics with positive scalar curvature, for which the maximal dimension of the singular set is $\frac{n}{2}-1$. Unfortunately, a $\gamma$-fractional curvature version of their classification results for locally conformally flat metrics of positive scalar curvature as quotients of the sphere by Kleinian groups is still far out of reach. Furthermore, complete metrics of positive constant $\gamma$-fractional curvature on $\R^n\setminus \Lambda$ have only been constructed in a few special cases with strong assumptions on the singular set; e.g., if $\Lambda$ is a finite set of points \cite{Ao-DelaTorre-Gonzalez-Wei:gluing}, or if $\Lambda$ is a prescribed submanifold of dimension $k<\frac{n}{2}-\gamma$ in \cite{ACDFGW,ACDFGW:survey} and in the critical dimension $k=\frac{n}{2}-\gamma$ in \cite{Chan-DelaTorre2}. More general constructions are available if one drops the completeness assumption and  considers weak solutions \cite{Ao-Chan-Gonzalez-Wei}.

A key example where solutions trivially exist is the round sphere $M=\S^n$ with a totally geodesic singular set $\Lambda=\S^k$, $0<k<n$. Using the conformal equivalence
\begin{equation*}
  \S^n\setminus\S^k\cong\S^{n-k-1}\times\H^{k+1},
\end{equation*}
we may pull back the standard product metric $\gp=\g_{\S^{n-k-1}}+\g_{\H^{k+1}}$ to obtain a complete metric on $\S^n\setminus \S^k$ with constant $\gamma$-fractional curvature equal to
\begin{equation}\label{eq:Qgamma-prodmetric}
  Q_\gamma(n,k) :=4^{\gamma}\, \frac{\Gamma\big(\frac{n+2\gamma}{4}\big)}{\Gamma\big(\frac{n-2\gamma}{4}\big)} \,\frac{\Gamma\big(\frac{n-2k+2\gamma}{4}\big)}{\Gamma\big(\frac{n-2k-2\gamma}{4}\big)},
\end{equation}
see \Cref{sec:trivialsolution}.
Note that $Q_\gamma(n,k)>0$ if $k<\frac{n}{2}-\gamma$, but the converse does not hold.

There is a fundamental dichotomy between the cases $k=1$ and $k>1$ regarding \emph{periodic} solutions on $\S^n\setminus\S^k$, that is, solutions which are invariant under a discrete cocompact group of conformal transformations. On the one hand, if $k=1$, then compact quotients of $\S^n\setminus\S^1\cong\S^{n-2}\times\H^2$ admit large families of pairwise non-conformal metrics, all of which are locally isometric to $\gp$. On the other hand, if $k>1$, then Mostow rigidity implies that there is a unique metric locally isometric to $\gp$ on each compact quotient of $\S^n\setminus\S^k\cong\S^{n-k-1}\times\H^{k+1}$.

Our first main result concerns the case $k=1$, in which multiplicity of solutions to the singular fractional Yamabe problem is established using bifurcation techniques:

\begin{mainthm}\label{mainthmA}
For all $n\geq4$ and $0<\gamma<c_n$, where $c_n$ is an increasing sequence asymptotic to $\frac{n}{2}-1$, there are uncountably many pairwise nonhomothetic periodic solutions to the singular fractional Yamabe problem on $\S^{n}\setminus \S^{1}$ that bifurcate from the trivial solution and have $\gamma$-fractional curvature arbitrarily close to $Q_\gamma(n,1)$.
\end{mainthm}

The definition of $(c_n)_{n\geq4}$ and approximate values for some small $n$ are given in \Cref{prop:cn}. In particular, $c_n>1$ for $n\geq5$, and $c_n>2$ for $n\geq7$, so Theorem~\ref{mainthmA} recovers earlier multiplicity results for solutions to the singular fractional Yamabe problem on $\S^n\setminus\S^1$ obtained for $\gamma=1$ if $n\geq5$ in~\cite{BPS}, and for $\gamma=2$ if $n\geq7$ in~\cite{bps-imrn}.
Similarly to these bifurcation results, the proof of Theorem~\ref{mainthmA} relies on detecting jumps of the Morse index along paths of metrics on compact quotients $\S^{n-2}\times\Sigma^{2}$ of $\S^n\setminus\S^1$ as the systole of the hyperbolic surface $\Sigma^2$ degenerates. This is achieved with a delicate analysis of the spectrum of the operator $P_\gamma^{\gp}$ making use of the assumptions $n\geq4$ and $\gamma<c_n$, see \Cref{prop:eigenvalues-P,propimp,prop:cn}.
While we expect the conclusion of Theorem~\ref{mainthmA} to hold for all $0<\gamma<\frac{n}{2}-1$, our methods fail if $\gamma$ is in the interval $\big[c_n,\frac{n}{2}-1\big)$, whose length decreases to zero as $n\nearrow+\infty$.

Our second main result deals with the case of higher dimensional singular sets:

\begin{mainthm}\label{mainthmB}
For all $n\geq 3$, given any $N\in \N$, there exists $\varepsilon>0$ such that if $\gamma \in \left(\frac{1}{2}-\varepsilon,\frac{1}{2}+\varepsilon\right)\cup (1-\varepsilon,1)$, then the singular fractional Yamabe problem on $\S^{n}\setminus\S^k$ has at least $N$ pairwise nonhomothetic periodic solutions for all $0\leq k<\frac{n}{2}-\gamma$.
\end{mainthm}

For $\gamma=1$ and $\gamma=2$, there are infinitely many nonhomothetic periodic solutions
on $\S^{n}\setminus\S^k$ for all $0\leq k<\frac{n}{2}-\gamma$, see \cite[Cor.~1.2]{bp-aif} and \cite[Cor.~C]{bps-imrn}, respectively. We conjecture that the same conclusion holds in the above setting, and also for $\gamma\in(1,1+\varepsilon)\cup \left(\frac{3}{2}-\varepsilon,\frac{3}{2}+\varepsilon\right)\cup(2-\varepsilon,2)$. The only missing ingredient to prove Theorem~\ref{mainthmB} for such values of $\gamma>1$ is an Aubin-type existence result for $Q_\gamma$, which seems to be technical but doable (see Remark \ref{remark:gamma}). The proof of Theorem~\ref{mainthmB} adapts the method from \cite{bp-aif,bps-imrn}, which overcomes the restriction imposed by Mostow rigidity combining the Aubin-type result with the existence of towers of finite-sheeted regular coverings of compact quotients of $\S^n\setminus\S^k$ whose volume is arbitrarily large. The reason why we need $\gamma$ to be near $\tfrac12$ or $1$ is that verifying the hypotheses in the Aubin-type existence result relies on the Positive Mass Theorem, which is currently only known to hold for such $\gamma$, see \Cref{lemma:yamabe}. As an alternative to the Positive Mass Theorem, one may assume the existence of a positive Green's function. Under this assumption, similar nonuniqueness results have been recently proved for the sixth-order constant $Q$-curvature problem ($\gamma=3$) in \cite{andrade2023nonuniqueness}, and for general conformally variational invariants of higher order in \cite{andrade2023general}.

Since stereographic projection gives a conformal equivalence $\S^n\setminus \S^0\cong \R^n\setminus\{0\}$, the case $k=0$ in Theorem~\ref{mainthmB} can be recovered using the construction of Delaunay-type solutions to the fractional Yamabe problem in \cite{delaunay}. Moreover, solutions with an isolated singularity have been recently shown to have infinite Morse index in \cite{CDR}.

This paper is organized as follows. In \Cref{section:problem}, we give the necessary background on the conformal fractional Laplacian and the $\gamma$-fractional curvature.
The spectrum of the conformal fractional Laplacian on $\S^{n-k-1}\times\Sigma^{k+1}$
is described in \Cref{sec:trivialsolution}, and Theorem~\ref{mainthmA} is proved in \Cref{section:thmA}. Finally, we prove Theorem~\ref{mainthmB} in \Cref{section:thmB}.

\subsection*{Acknowledgements}
R.\ G.\ Bettiol acknowledges financial support from the National Science Foundation, through NSF grant DMS-1904342 and NSF CAREER grant DMS-2142575. M.~Gonz\'alez acknowledges financial support from the Spanish  Government: PID2020-113596GB-I00, RED2018-102650-T funded by MCIN/AEI/ 10.13039/ 501100011033, and the ``Severo Ochoa Programme for Centers of Excellence in R\&D'' (CEX2019-000904-S).
We would like to thank the anonymous referee for the careful reading of our paper and thoughtful suggestions for improvements.

\section{Fractional Yamabe Problem}\label{section:problem}

In this section, we recall the definitions of conformal fractional Laplacian and fractional curvature, closely following \cite[Sec.~2]{Gonzalez:survey}.

\subsection{Fractional Laplacian and fractional curvature}
Let $n\geq4$ and consider a conformally compact $(n+1)$-dimensional Poincar\'e--Einstein manifold $(M^{+},\g^{+})$ with conformal infinity $(M,[\g])$. Near the boundary, the metric $\g^{+}$ takes the form
\begin{equation}\label{eq:g+canform}
\g^{+}=\frac{1}{\rho^{2}}\big(\dd\rho^{2}+\h_{\rho}\big),
\end{equation}
where $\rho$ is a defining function for $M$ in $M^+$, and $\h_{\rho}$ is a $1$-parameter family of metrics such that $\h_\rho|_{\rho=0}=\g$.
Fix $\gamma \in \left(0,\frac{n}{2}\right)$, $\gamma\notin\N$, and consider the eigenvalue problem on $(M^{+},\g^+)$ given by
\begin{equation}\label{eq:PDE}
-\Delta_{\g^{+}}W-\left(\tfrac{n^{2}}{4}-\gamma^2\right)W=0.
\end{equation}
Supposing that $\tfrac{n^{2}}{4}-\gamma^2$ is not an $L^2$-eigenvalue (we will also assume, for later purposes, that $\lambda_1(-\Delta_{\g^+})>\frac{n^2}{4}-\gamma^2$), this PDE has a unique solution
\begin{equation*}
W=W_{1}\,\rho^{\frac{n}{2}-\gamma}+W_{2}\,\rho^{\frac{n}{2}+\gamma},
\end{equation*}
under the Dirichlet condition $W_{1}|_{\rho=0}=w$, for $W_1,W_2$ smooth solutions up to $\overline {M^+}$.
\begin{definition}\label{def:fracLap-fraccurv}
The \emph{conformal $\gamma$-fractional Laplacian} $P_{\gamma}^{\g}$ is the operator
\begin{equation*}
\phantom{, \quad\text{ where } d_{\gamma}=4^{\gamma}\frac{\Gamma(\gamma)}{\Gamma(-\gamma)}.}
P_{\gamma}^{\g}w=d_{\gamma}\,W_{2}|_{\rho=0}, \quad\text{ where } d_{\gamma}=4^{\gamma}\frac{\Gamma(\gamma)}{\Gamma(-\gamma)},
\end{equation*}
and the \emph{$\gamma$-fractional curvature} of $(M,\g)$ is its zeroth order term $Q_{\gamma}^{\g}:=P_{\gamma}^{\g}(1)$.
\end{definition}

The scattering operators $P_{\gamma}^{\g}$ form a $1$-parameter meromorphic family of  self-adjoint pseudo-differential operators on $M$, with poles at integer values of $\gamma$, whose principal symbol is the same as that of $(-\Delta_\g)^\gamma$. These operators are conformally covariant, in the sense that
a conformal change of metric $\g_u=u^{\frac{4}{n-2\gamma}}\g$ leads to
\begin{equation}\label{eq:conf-cov}
  P_\gamma^{\g_u}(\cdot)=u^{-\frac{n+2\gamma}{n-2\gamma}}P_\gamma^\g(u\,\cdot).
\end{equation}
Furthermore, these operators are \emph{non-local}, in the sense that they depend not only on $(M,\g)$, but also on its Poincar\'e--Einstein filling $(M^+,\g^+)$.

The normalizing constant $d_\gamma$ ensures that $P_\gamma^\g$ can be extended to integer values of $\gamma$ with an appropriate residue formula~\cite{graham-zworski}. This extension gives rise to \emph{local} differential operators on $(M,\g)$ that also satisfy \eqref{eq:conf-cov} and, up to multiplication by a (positive) dimensional constant, coincide with the conformally covariant powers of the Laplacian known as GJMS operators. In particular, $P_{1}^{\g}=-\Delta_\g+\frac{n-2}{4(n-1)}R_\g$ is the conformal Laplacian, and $P_{2}^{\g}=(-\Delta_\g)^2+\dots$ is the Paneitz operator.

By \eqref{eq:conf-cov}, the fractional curvature of a conformal metric $\g_u=u^{\frac{4}{n-2\gamma}}\g$ is:
\begin{equation*}
Q_{\gamma}^{\g_{u}}=u^{-\frac{n+2\gamma}{n-2\gamma}}\, P_{\gamma}^{\g}u.
\end{equation*}
Thus, solutions $\g_u$ to the fractional Yamabe problem on $(M,\g)$ correspond to positive solutions $u\colon M\to\R$ to the nonlinear eigenvalue problem
\begin{equation}\label{eq:P-eigenvalue-prob}
\phantom{, \quad \lambda \in\R}
P_{\gamma}^{\g}u=\lambda \, |u|^{\frac{4\gamma}{n-2\gamma}}u, \quad \lambda \in\R.
\end{equation}
Similarly, solutions $\g_u$ to the singular fractional Yamabe problem on $(M,\g)$ with singular locus $\Lambda\subset M$ correspond to positive solutions to \eqref{eq:P-eigenvalue-prob} such that $u\nearrow+\infty$ sufficiently fast at $\Lambda$ so that $(M\setminus\Lambda,\g_u)$ is complete. The main idea in \cite{MQ}, as in the classical Yamabe problem for $\gamma=1$,  is to characterize solutions to \eqref{eq:P-eigenvalue-prob} variationally as critical points of the energy functional
\begin{equation}\label{eq:energy}
E_\gamma^{\g}\colon W^{\gamma,2}(M)\to\R, \qquad E_\gamma^{\g}(u):= \int_{M} u\,P_{\gamma}^{\g}u\;\mathrm{vol}_\g,
\end{equation}
where $\mathrm{vol}_\g$ is the volume element of $\g$, subject to the constraint
\begin{equation}\label{eq:constraint}
\int_{M} |u|^{\frac{2n}{n-2\gamma}}\;\mathrm{vol}_\g=\Vol(M,\g),
\end{equation}
and look for minimizers of $E_\gamma^{\g}$, which turn out to be positive (see Proposition~\ref{prop:aubin-ineq}).

If $u\equiv1$ is one such critical point, i.e., $Q_\gamma^\g$ is constant, then its Morse index and nullity, also referred to as Morse index and nullity of $\g$, are the number (counted with multiplicity) of negative eigenvalues and nullity of the Jacobi operator
\begin{equation}\label{eq:jacobi}
J_\gamma^{\g}=P_{\gamma}^{\g}-\tfrac{n+2\gamma}{n-2\gamma}\,Q_{\gamma}^{\g},
\end{equation}
as an unbounded self-adjoint operator on $L^2_0(M)=\big\{v\in L^2(M) : \int_M v\;\mathrm{vol}_\g=0\big\}$. Finally, note that $J_\gamma^\g\colon W^{\gamma,2}(M)\to W^{-\gamma,2}(M)$ is a Fredholm operator of index zero. The fact that $J_\gamma^\g$ is Fredholm follows by a standard argument from the compact embedding of $W^{\gamma,2}(M)$ in $L^{2}(M)$ and the fact that $P_{\gamma}^\g$ is an elliptic operator with leading term $(-\Delta_\g)^{\gamma}$.

\subsection{Existence of minimizing solutions}
In order to seek \emph{minimizing} solutions to the $\gamma$-fractional Yamabe problem, i.e., positive minimizers of \eqref{eq:energy} satisfying \eqref{eq:constraint}, one may extend the notion of Yamabe constant beyond the classical case $\gamma=1$ as follows:

\begin{definition}\label{def:gamma-Yamabe-const}
The \emph{$\gamma$-Yamabe constant} of the conformal class $[\g]$ is given by
\begin{equation}\label{yam}
\Lambda_{\gamma}(M,[\g]):=
\inf_{\substack{u\in W^{\gamma,2}(M)}}\frac{ E_\gamma^\g(u) }{\Big(\int_{M} u^{\frac{2n}{n-2\gamma}}\;\mathrm{vol}_{\g}\Big)^{\frac{n-2\gamma}{n}}} =\inf_{\g_u\in [\g]}\frac{\int_{M} Q_{\gamma}^{\g_u}\;\mathrm{vol}_{\g_u}}{\Vol(M,\g_u)^{\frac{n-2\gamma}{n}}}.
\end{equation}
\end{definition}

The following Aubin-type existence result for the fractional Yamabe problem with $\gamma\in(0,1)$ was proven in \cite{MQ,CC}.

\begin{proposition}\label{prop:aubin-ineq}
For a closed $n$-dimensional manifold $(M,\g)$ as above, we have
\begin{equation}\label{eq:Aubin-ineq}
-\infty<\Lambda_{\gamma}(M,[\g])\leq \Lambda_{\gamma}(\S^{n},[\g_{\S^{n}}]),
\end{equation}
where $\gamma\in(0,1)$ and $\g_{\S^n}$ is the unit round metric on $\S^{n}$.
Moreover, if
\begin{equation*}
\Lambda_{\gamma}(M,[\g])< \Lambda_{\gamma}(\S^{n},[\g_{\S^n}]),
\end{equation*}
then the infimum \eqref{yam} is achieved by a positive function $u\in W^{\gamma,2}(M)$; in other words, $\g_u$ is a minimizing solution to the $\gamma$-fractional Yamabe problem on $(M,\g)$. In addition, $u\in C^\infty$.
\end{proposition}

\begin{remark}\label{remark:gamma}
We expect that a similar result holds also for exponents $\gamma\in(1,2)$, following the ideas of \cite{HY,Gursky-Malchiodi} for the $Q$-curvature, since the main ingredient is a maximum principle to show that a minimizer $u$ of \eqref{yam} is strictly positive. For $\gamma\in(1,2)$, this maximum principle was proved in \cite[Thm~1.3]{CC}.
\end{remark}

\begin{remark}
Similarly to the classical case $\gamma=1$, a rigidity statement if equality holds in \eqref{eq:Aubin-ineq} is known if $\gamma=\tfrac12$, or $\gamma=2$ and $R_\g\geq0$, $Q_{\gamma}^{\g}\geq 0$, and $Q_{\gamma}^{\g}\not\equiv 0$. Namely, in these cases, $\Lambda_{\gamma}(M,[\g]) = \Lambda_{\gamma}(\S^{n},[\g_{\S^n}])$ if and only if $(M,\g)$ is conformally equivalent to $(\S^n,\g_{\S^n})$, see \cite{Esc,HY}.
\end{remark}

\section{On the trivial solution}\label{sec:trivialsolution}

In this section, we examine spectral properties of the conformal fractional Laplacian on compact quotients of $\S^n\setminus\S^k\cong \S^{n-k-1}\times \H^{k+1}$, using results from \cite{ACDFGW}.

First, we recall the following elementary fact:

\begin{lemma}\label{lemma:conf-equiv}
The complement $\S^{n}\setminus \S^{k}$ of a totally geodesic subsphere $\S^k$ in the unit round sphere $\S^n$, $0<k<n$, is conformally equivalent to $\S^{n-k-1}\times \H^{k+1}$, endowed with the product metric $\gp$ of metrics of constant curvature $1$ and $-1$.
\end{lemma}

\begin{proof}
Stereographic projection using a point in $\S^k$ gives a conformal equivalence from $\S^{n}\setminus \S^{k}$ to $\R^{n}\setminus \R^{k}=(\R^{n-k}\setminus \{0\})\times \R^k$ endowed with the flat metric, which can be written in cylindrical coordinates as $\dd r^2+r^2\dd \theta^2+\dd y^2$, where $(r,\theta)$ are polar coordinates in $\R^{n-k}\setminus \{0\}$ and $y$ is the coordinate in~$\R^k$. Dividing by $r^2$, one obtains $\dd\theta^2+\frac{1}{r^2}(\dd r^2+\dd y^2)$, which is precisely $\gp=\g_{\S^{n-k-1}}+\g_{\H^{k+1}}$.
\end{proof}

Denote by $\phi\colon \S^n\setminus\S^k\to \S^{n-k-1}\times \H^{k+1}$ the conformal equivalence in \Cref{lemma:conf-equiv}.
For all $1\leq k < \frac{n}{2}-\gamma$, we call the pullback metric $\gtriv=\phi^*\gp$ the \emph{trivial} solution to the singular fractional Yamabe problem on $\S^n\setminus \S^k$, since it is a complete metric that, as we will now show, has constant (positive) fractional curvature, see \eqref{eq:trivial-fractional-curvature}.

The Poincar\'e--Einstein manifold $(M^+,\g^+)$ with conformal infinity $(M,\g)$ given by $(\S^n\setminus \S^k,\gtriv)\cong (\S^{n-k-1}\times \H^{k+1},\gp)$ is clearly $\H^{n+1}\setminus \H^{k+1}$, see \Cref{fig:ball}.
\begin{figure}[!ht]
\centering
\vspace{-.5cm}
\begin{tikzpicture}[scale=0.5,every node/.style={minimum size=1cm}]
\def\R{4} 
\def\angEl{25} 
\def\angAz{-100} 

\pgfmathsetmacro\H{\R*cos(\angEl)} 
    \LongitudePlane[nzplane]{\angEl}{-86}
\LatitudePlane[equator]{\angEl}{0}
\fill[ball color=white!10] (0,0) circle (\R);
\coordinate (O) at (0,0);
\coordinate[mark coordinate, red] (N) at (0,\H);
\coordinate[mark coordinate, red] (S) at (0,-\H);

\DrawLongitudeCircle[\R]{5}
\DrawLongitudeCircle[\R]{35}
\DrawLongitudeCircle[\R]{65}
\DrawLongitudeCircle[\R]{95}
\DrawLongitudeCircle[\R]{125}
\DrawLongitudeCircle[\R]{155}

\node[above] at (N) {$\color{red}\S^k$};
    \draw[-,dashed, thick, red] (N) -- (S);

\DrawLatitudeCircleblue[\R]{0} 
\node at (1.2,0) {$\color{red} \mathds H^{k+1}$};
\node[right] at (\R+0.5,0.15) {$\color{blue} \mathds S^{n-k-1}$};
\node[left] at (-\R-0.5,0.15) {$\phantom{\mathds S^{n-k-1}}$};
\coordinate[mark coordinate, red] (N) at (0,\H);
\coordinate[mark coordinate, red] (S) at (0,-\H);
\node at (-1.5,.8) {$\mathds H^{n+1}$};
\node at (-4.3,2.2) {$\mathds S^{n}$};
\end{tikzpicture}
\caption{Schematic depiction of $\S^n$ as the join of $\S^k$ and $\S^{n-k-1}$. Note that
$\S^{n}\setminus\S^k$ is conformal to the (trivial) normal bundle of $\S^{n-k-1}\subset\S^n$ and its Poincar\'e--Einstein filling is $\H^{n+1}\setminus \H^{k+1}$.}\label{fig:ball}
\vspace{-.2cm}
\end{figure}
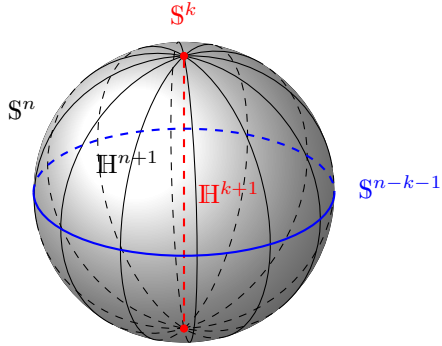

More precisely, writing $\g^{+}$ in the form \eqref{eq:g+canform}, we have
\begin{equation*}
\textstyle
\g^{+}=\frac{1}{\rho^{2}}\left(\dd\rho^{2}+\left(1-\frac{\rho^{2}}{4}\right)^{2}\g_{\S^{n-k-1}}+\left(1+\frac{\rho^{2}}{4}\right)^{2}\g_{\H^{k+1}}\right),
\end{equation*}
and, changing variables to $\sigma=\ln \frac{2}{\rho}$, with $\rho \in (0,2)$, the above becomes
\begin{equation}\label{eq:g+hyp}
\g^{+}=\dd\sigma^{2}+\sinh^{2}(\sigma)\,\g_{\S^{n-k-1}}+\cosh^{2}(\sigma)\,\g_{\H^{k+1}},
\end{equation}
which is the usual doubly warped product presentation of the hyperbolic metric on $\H^{n+1}\setminus \H^{k+1}$, where $\sigma>0$ is the hyperbolic distance to $\H^{k+1}$.

Using \eqref{eq:g+hyp}, we can write equation \eqref{eq:PDE} as follows, cf.~\cite[Eqn.~(3.26)]{ACDFGW}:
\begin{equation}\label{eq:PDE-coord}
\tfrac{\partial^2}{\partial {\sigma}^{2}} W+h(\sigma) \,\tfrac{\partial}{\partial {\sigma}} W+\tfrac{1}{\sinh^{2}\sigma}\, \Delta_{\S^{n-k-1}}W+\tfrac{1}{\cosh^{2}\sigma}\,\Delta_{\H^{k+1}}W+\left(\tfrac{n^{2}}{4}-\gamma^{2}\right)W=0,
\end{equation}
where $h(\sigma)=\tfrac{\partial}{\partial {\sigma}} \ln \big(\!\sinh^{n-k-1}(\sigma)\cosh^{k+1}(\sigma)\big)$. In order to compute the spectrum of the conformal fractional Laplacian on compact quotients of $\S^{n-k-1}\times\H^{k+1}$, we proceed as in \cite[Thm.~3.5]{ACDFGW}, using separation of variables on \eqref{eq:PDE-coord}.

On the one hand, we consider the spherical harmonic decomposition for $\S^{n-k-1}$. Namely, let $E_m$ be the eigenfunctions for the Laplacian on $\S^{n-k-1}$, with eigenvalue $\mu_m=m(m+n-k-2)$, repeated according to multiplicity.
Then, functions $w$ on $\S^{n-k-1}\times \H^{k+1}$ may be decomposed as $w=\sum_{m} w_m E_m$, where $w_m=w_m(\zeta)$ for $\zeta\in\H^{k+1}$.
On the other hand, we use the Fourier--Helgason transform $\widehat{\cdot}$ on $\H^{k+1}$.
The only property we shall need is that
\begin{equation}\label{Helgason}
\widehat{\Delta_{\H^{k+1}} w_m}=-\big(\xi^2+\tfrac{k^2}{4}\big)\,\widehat{w_m},
\end{equation}
we refer to \cite[Appendix]{ACDFGW} for details, and recall Theorem~3.5 in that paper:

\begin{proposition}\label{prop:symbol}
 Fix $0<\gamma<\tfrac{n}{2}$,  $1\leq k<\frac{n}{2}-\gamma$, and let $P^{(m)}_{\gamma}$ be the projection of the operator $P_\gamma^{\gp}$ over each eigenspace $\langle E_m\rangle$. Then
$$\widehat{P_\gamma^{(m)} (w_m)}=\Theta_\gamma^{(m)}(\xi) \,\widehat{w_m},$$
and this Fourier symbol is given by
\begin{equation}\label{symbol}
\Theta_\gamma^{(m)}(\xi)=4^{\gamma}\frac{\Gamma\big(\frac{1+\gamma}{2}
+\frac{a_m+\xi i}{2}\big)}{\Gamma\big(\frac{1-\gamma}{2}+\frac{a_m+\xi i}{2}\big)}\frac{
\Gamma\big(\frac{1+\gamma}{2}+\frac{a_m-\xi i}{2}\big)}{\Gamma\big(\frac{1-\gamma}{2}+\frac{a_m-\xi i}{2}\big)},\quad \xi\in\R,
\end{equation}
where
\begin{equation}\label{am}
 a_m=\sqrt{\mu_m+ \big(\tfrac{n-k-2}{2}\big)^2}.
 \end{equation}
\end{proposition}

The corresponding spectrum on compact quotients can be extracted as follows:

\begin{proposition}\label{prop:eigenvalues-P}
Consider $\S^{n-k-1}\times\Sigma^{k+1}$ endowed with the product metric $\gp$, where $\Sigma^{k+1}=\H^{k+1}/\pi$ is a closed hyperbolic manifold.
  Let $0=\lambda_0<\lambda_1\leq\dots\leq\lambda_\ell\leq\dots\nearrow+\infty$ be the Laplace eigenvalues on $\Sigma^{k+1}$, repeated according to their multiplicity.
If $1\leq k<\frac{n}{2}-\gamma$, the spectrum of $P_\gamma^{\gp}$ consists of the discrete set of eigenvalues
\begin{equation}\label{eq:eigenvaluesTheta}
\phantom{\quad \ell\in\N}
\Theta_{m,\ell}=4^\gamma\,\frac{\Gamma\big(\frac{1+\gamma}{2}+\frac{a_m+b_\ell i}{2} \big)}{\Gamma\big(\frac{1-\gamma}{2}+\frac{a_m+b_\ell i}{2} \big)}\,\frac{\Gamma\big(\frac{1+\gamma}{2}+\frac{a_m-b_\ell i}{2} \big)}{\Gamma\big(\frac{1-\gamma}{2}+\frac{a_m-b_\ell i}{2} \big)}, \quad m,\ell\in\N\cup\{0\},
\end{equation}
where $a_m\geq 0$ is given in \eqref{am} and $b_\ell\in\C$ is
\begin{equation}\label{eq:ambl}
  \textstyle
 b_\ell=\sqrt{\lambda_\ell-\big(\frac{k}{2}\big)^2}.
\end{equation}
The corresponding eigenspace of $P_\gamma^{\gp}$ is spanned by products of eigenfunctions of the Laplacian on $\S^{n-k-1}$ and on $\Sigma^{k+1}$ with eigenvalues $\mu_m$ and $\lambda_\ell$, respectively.
\end{proposition}

\begin{proof}
This follows from Proposition \ref{prop:symbol}, applied to functions $w_m\colon \H^{k+1}\to\R$ that are invariant under the discrete cocompact group $\pi$ of isometries of $\H^{k+1}$ and $\xi=b_\ell$, keeping in mind \eqref{Helgason}.
\end{proof}

\begin{remark}
 In \eqref{eq:ambl}, we use the standard convention that, if $\lambda_\ell<\big(\frac{k}{2}\big)^2$, then
 \begin{equation*}
    \textstyle \sqrt{\lambda_\ell-\big(\frac{k}{2}\big)^2} :=i \sqrt{\big(\frac{k}{2}\big)^2-\lambda_\ell}.
 \end{equation*}
\end{remark}

\begin{corollary}
For all $1\leq k<\frac{n}{2}-\gamma$, the fractional curvature of $(\S^n\setminus\S^k,\gtriv)$ is
\begin{equation}\label{eq:trivial-fractional-curvature}
Q_\gamma^{\gtriv} = \Theta_{0,0} = 4^{\gamma}\, \frac{\Gamma\big(\frac{n+2\gamma}{4}\big)}{\Gamma\big(\frac{n-2\gamma}{4}\big)} \,\frac{\Gamma\big(\frac{n-2k+2\gamma}{4}\big)}{\Gamma\big(\frac{n-2k-2\gamma}{4}\big)}.
\end{equation}
\end{corollary}

\begin{proof}
By \Cref{def:fracLap-fraccurv} and \Cref{prop:eigenvalues-P}, we have
$Q_\gamma^{\gtriv}= P_\gamma^{\gp}(1)=\Theta_{0,0}$, since $(\S^{n-k-1}\times\Sigma^{k+1},\gp)$ is locally isometric to $(\S^{n-k-1}\times\H^{k+1},\gp)$, which, by definition, is isometric to $(\S^n\setminus\S^k,\gtriv)$. Formula \eqref{eq:trivial-fractional-curvature} then follows from \eqref{eq:eigenvaluesTheta}, using that $a_0=\frac{n-k-2}{2}$ and $b_0=\frac{k}{2}i$.
\end{proof}

We shall also denote the above quantity \eqref{eq:trivial-fractional-curvature} by $Q_\gamma(n,k)$, as in \eqref{eq:Qgamma-prodmetric}.

\begin{remark}
Clearly, the only eigenvalues $\Theta_{m,\ell}$ that are independent of the cocompact group $\pi$ such that $\Sigma^{k+1}=\H^{k+1}/\pi$ are those of the form $\Theta_{m,0}$. 
Eigenfunctions corresponding to $\Theta_{m,0}$ are compositions of the projection $\S^{n-k-1}\times \Sigma^{k+1}\to \S^{n-k-1}$ onto the first factor with degree $m$ spherical harmonics on $\S^{n-k-1}$. 
\end{remark}

Let us observe that $\Theta_{m,\ell}$ are positive real numbers, for all $m,\ell\in\N\cup\{0\}$. Indeed, from \eqref{eq:eigenvaluesTheta} and \eqref{eq:ambl}, we have that $\Theta_{m,\ell}=\Theta(a_m,b_\ell)$, where
\begin{equation}\label{eq:theta-fct}
\Theta(a,b):=4^\gamma\,\frac{\Gamma\big(\frac{1+\gamma}{2}+\frac{a+b i}{2} \big)}{\Gamma\big(\frac{1-\gamma}{2}+\frac{a+b i}{2} \big)}\,\frac{\Gamma\big(\frac{1+\gamma}{2}+\frac{a-b i}{2} \big)}{\Gamma\big(\frac{1-\gamma}{2}+\frac{a-b i}{2} \big)}.
\end{equation}
Clearly, $\Theta(a,b)>0$ for all $a\geq a_0=\frac{n-k-2}{2}$ and $b$ which is either purely imaginary and $|b|\leq \frac{k}{2}$, or real and $b\geq0$, exactly as $a_m$ and $b_\ell$ in \eqref{am},\eqref{eq:ambl}, since
$\Gamma(\overline{z})=\overline{\Gamma(z)}$, and $1\leq k<\frac{n}{2}-\gamma$ implies that $\operatorname{Re}\left[\tfrac{1\pm \gamma}{2}+\frac{a\pm b i}{2}\right]>0$,
i.e., all complex numbers $z$ on which $\Gamma(z)$ is evaluated in \eqref{eq:theta-fct} have positive real part.

Moreover, the function $\Theta(a,b)$ enjoys the following monotonicity properties:

\begin{lemma}\label{lemma:monotonicity}
Given $1\leq k<\frac{n}{2}-\gamma$, the function $\Theta(a,b)$ in \eqref{eq:theta-fct} satisfies:
\begin{enumerate}[label = \rm (\roman*),leftmargin=*,itemsep=2pt]
  \item $\frac{\partial \Theta}{\partial a}(a,\beta i) > 0$ for all $a\geq a_0$, and $0<\beta \leq \tfrac{k}{2}$,
  \item  $\frac{\partial \Theta}{\partial \beta}(a,\beta i) < 0$ for all $a\geq a_0$, and $0<\beta \leq \tfrac{k}{2}$,
  \item $\frac{\partial \Theta}{\partial b}(a,b) > 0$ for all $a\geq a_0$, and $b > 0$.
\end{enumerate}
\end{lemma}

\begin{proof}
Since $\Theta(a,b)>0$ for the values of $a$ and $b$ under consideration, it suffices to show that the logarithmic derivatives of $\Theta(a,b)$ satisfy the inequalities in (i)--(iii).
Computing these in terms of the \emph{digamma function} $\psi(z)=\frac{\Gamma'(z)}{\Gamma(z)}$, we obtain:
\begin{equation}\label{eq:der-theta}
\begin{aligned}
\frac{\partial}{\partial a} \log \Theta(a,b) &=\tfrac12\Big( \psi\big(\tfrac{1+\gamma}{2}+\tfrac{a+b i}{2}\big) +  \psi\big(\tfrac{1+\gamma}{2}+\tfrac{a-b i}{2}\big) \\
&\qquad -\psi\big(\tfrac{1-\gamma}{2}+\tfrac{a+b i}{2}\big) -\psi\big(\tfrac{1-\gamma}{2}+\tfrac{a-b i}{2}\big)\Big),\\[2pt]
\frac{\partial}{\partial b} \log \Theta(a,b) &=\tfrac{i}{2}\Big(  \psi\big(\tfrac{1+\gamma}{2}+\tfrac{a+b i}{2}\big) - \psi\big(\tfrac{1+\gamma}{2}+\tfrac{a-b i}{2}\big) \\
&\qquad -\psi\big(\tfrac{1-\gamma}{2}+\tfrac{a+b i}{2}\big)+\psi\big(\tfrac{1-\gamma}{2}+\tfrac{a-b i}{2}\big) \Big).
\end{aligned}
\end{equation}

As $\gamma>0$, recall that if $z=x+y i$, where $x=\operatorname{Re} z>0$ and $y=\operatorname{Im} z\geq 0$, then
\begin{equation}\label{eq:difference-psis}
\begin{aligned}
  \psi(z+\gamma)-\psi(z)&=\sum_{j=0}^{+\infty} \frac{1}{j+z}-\frac{1}{j+z+\gamma}=\sum_{j=0}^{+\infty} \frac{\gamma}{(j+z)(j+z+\gamma)}\\
  &= \gamma\sum_{j=0}^{+\infty} \frac{\big((j+x)(j+x+\gamma)-y^2 \big) - \big(y( 2j+2x+\gamma) \big) i }{\big((j+x)^2+y^2\big)\big((j+x+\gamma)^2+y^2\big)}.
\end{aligned}
\end{equation}
We apply \eqref{eq:difference-psis} with $z=\frac{1-\gamma}{2}+\frac{a\pm bi}{2}$.
First, if $b=\beta i$ is purely imaginary and $0<\beta\leq \frac{k}{2}$, then using \eqref{eq:difference-psis} on \eqref{eq:der-theta}, and writing $x_\pm =\frac{1-\gamma}{2}+\frac{a\pm \beta}{2}>0$, gives:
\begin{equation*}
  \frac{\partial}{\partial a} \log \Theta(a,b) = \frac\gamma2 \sum_{j=0}^{+\infty} \frac{  1}{(j+x_+)(j+x_++\gamma) } +  \frac{ 1 }{(j+x_-)(j+x_-+\gamma)}  >0,
\end{equation*}
which proves (i). Moreover, from \eqref{eq:der-theta}, we have that if $b=\beta i$, then
\begin{align*}
  \frac{\partial}{\partial \beta } \log \Theta(a,\beta i)  &= \tfrac12\Big(\psi\big(\tfrac{1+\gamma}{2} +\tfrac{a+\beta}{2}\big) -\psi\big(\tfrac{1-\gamma}{2} +\tfrac{a+\beta}{2}\big) \\
&\qquad -\Big( \psi\big(\tfrac{1+\gamma}{2} +\tfrac{a-\beta}{2}\big) -  \psi\big(\tfrac{1- \gamma}{2} +\tfrac{a-\beta}{2}\big) \Big)\Big)\\
&=  \frac\gamma2 \sum_{j=0}^{+\infty} \frac{  1}{(j+x_+)(j+x_++\gamma) } - \frac{ 1 }{(j+x_-)(j+x_-+\gamma)} \\
&=  - \frac\gamma2 \sum_{j=0}^{+\infty} \frac{ (x_+-x_-)(2j+x_+ + x_- +\gamma) }{(j+x_+)(j+x_-)(j+x_++\gamma) (j+x_-+\gamma)} <0,
\end{align*}
since $x_+ - x_- = \beta >0$, which proves (ii).

Second, if $b > 0$, then \eqref{eq:der-theta} simplifies as:
\begin{equation}\label{eq:der-theta-real} 
\begin{aligned}
\tfrac{\partial}{\partial a} \log \Theta(a,b) &= \phantom{-}\operatorname{Re}\left[ \psi\big(\tfrac{1+\gamma}{2}+\tfrac{a+b i}{2}\big) -\psi\big(\tfrac{1-\gamma}{2}+\tfrac{a+b i}{2}\big) \right],\\[2pt]
\tfrac{\partial}{\partial b} \log \Theta(a,b) &= - \operatorname{Im}\left[ \psi\big(\tfrac{1+\gamma}{2}+\tfrac{a+b i}{2}\big) -\psi\big(\tfrac{1-\gamma}{2}+\tfrac{a+b i}{2}\big) \right].
\end{aligned}
\end{equation}
Writing $x=\operatorname{Re} z =\frac{1-\gamma}{2}+\frac{a}{2}$ and $y=\operatorname{Im} z=\frac{b}{2}$, we have from \eqref{eq:difference-psis} and \eqref{eq:der-theta-real}  that
\begin{align*}
  \frac{\partial}{\partial b} \log \Theta(a,b)  &=\gamma  \sum_{j=0}^{+\infty}  \frac{  y( 2j+2x+\gamma) }{\big((j+x)^2+y^2\big)\big((j+x+\gamma)^2+y^2\big)} >0, \text{ if } b>0,
\end{align*}
which shows (iii).
\end{proof}

\begin{proposition}\label{prop:monotonicity-theta}
The eigenvalues $\Theta_{m,\ell}$ in \eqref{eq:eigenvaluesTheta} satisfy:
\begin{equation*}
\Theta_{m+1,0} > \Theta_{m,0}, \quad \text{ and }\quad \Theta_{m,\ell+1} \geq \Theta_{m,\ell}.
\end{equation*}
\end{proposition}

\begin{proof}
Since $m\mapsto a_m$ is monotonically increasing, applying \Cref{lemma:monotonicity} (i) with $\beta=\frac{k}{2}$, we have that $\Theta(a_{m+1},b_0) > \Theta(a_{m},b_0)$, i.e., $\Theta_{m+1,0} > \Theta_{m,0}$, which proves the first inequality. For the second inequality, recall that $\ell\mapsto\lambda_\ell$ is monotonically nondecreasing and unbounded. Thus, there exists $\ell_*\in\N$ such that $\ell\mapsto b_\ell$ is purely imaginary with monotonically nonincreasing imaginary part for $0\leq \ell < \ell_*$, and is real and monotonically nondecreasing for $\ell \geq \ell_*$. \Cref{lemma:monotonicity} (ii) and (iii) imply that $\Theta(a_m,b_{\ell+1}) \geq \Theta(a_m,b_{\ell})$ for all $0\leq \ell <\ell_*$ and $\ell\geq \ell_*$, respectively. It also follows from \Cref{lemma:monotonicity} (ii) and (iii) that $\Theta(a_m,b_{\ell_*}) \geq \Theta(a_m,b_{\ell_*-1})$, since $\Theta(a,b)$ is continuous. Thus, altogether, $\Theta_{m,\ell+1} \geq \Theta_{m,\ell}$ for all $m,\ell\in\N\cup\{0\}$.
\end{proof}

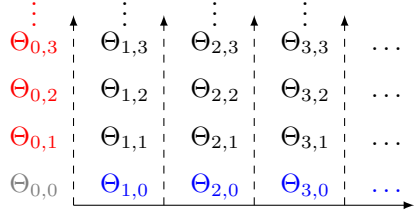
\begin{figure}[!ht]
\centering
\renewcommand{\arraystretch}{1.5}
\begin{tabular}{ccccc}
    $\color{red} \stackrel{\vdots}{\Theta_{0,3}}$\;\;\raisebox{10pt}{\tikzmark{03}} & $\stackrel{\vdots}{\Theta_{1,3}}$\;\;\raisebox{10pt}{\tikzmark{13}} & $\stackrel{\vdots}{\Theta_{2,3}}$\;\;\raisebox{10pt}{\tikzmark{23}}  & $\stackrel{\vdots}{\Theta_{3,3}}$\;\;\raisebox{10pt}{\tikzmark{33}}  & $\dots$ \\
    $\color{red}\Theta_{0,2}$\;\;              & $\Theta_{1,2}$\;\;               & $\Theta_{2,2}$\;\;               & $\Theta_{3,2}$\;\;  & $\dots$  \\
    $\color{red}\Theta_{0,1}$\;\;              & $\Theta_{1,1}$\;\;               & $\Theta_{2,1}$\;\;               & $\Theta_{3,1}$\;\;  & $\dots$  \\
    $\color{gray}\Theta_{0,0}$\;\;\raisebox{-5pt}{\tikzmark{00}} & $\color{blue}\Theta_{1,0}$\;\;\raisebox{-5pt}{\tikzmark{10}}  & $\color{blue}\Theta_{2,0}$\;\;\raisebox{-5pt}{\tikzmark{20}}  & $\color{blue}\Theta_{3,0}$\;\;\raisebox{-5pt}{\tikzmark{30}}  & $\color{blue}\dots$\raisebox{-5pt}{\tikzmark{40}}  \\
\end{tabular}

\begin{tikzpicture}[overlay,remember picture, shorten >=-3pt]
\draw[->,dashed] (pic cs:00) -- (pic cs:03);
\draw[->,dashed] (pic cs:10) -- (pic cs:13);
\draw[->,dashed] (pic cs:20) -- (pic cs:23);
\draw[->,dashed] (pic cs:30) -- (pic cs:33);
\draw[->] (pic cs:00) -- (pic cs:40);
\end{tikzpicture}
\vspace{-.1cm}
\caption{Dashed arrows indicate nondecreasing directions on the table, while the continuous arrow indicates an increasing direction.}\label{fig:ordering-thetas}
\end{figure}
\renewcommand{\arraystretch}{1}

\section{\texorpdfstring{Multiplicity of solutions on $\S^n\setminus \S^1$}{Multiplicity of solutions singular along a circle}}\label{section:thmA}

Throughout this section, we assume $n\geq4$ and $k=1$, and denote by $\Sigma^2=\H^2/\pi$ a closed hyperbolic surface of genus $\gen(\Sigma^2)\geq2$. We shall collectively denote by $\g=\gp=\g_{\S^{n-2}} + \h$ all the product metrics on $M=\S^{n-2}\times \Sigma^2$ in which $\h\in\mathcal H(\Sigma^2)$ is a  smooth hyperbolic metric on $\Sigma^2$, but we write
\begin{equation*}
  \phantom{, \qquad m\in\N\cup\{0\}, \, \ell\in\N}
  \Theta_{m,\ell}(\h)=\Theta_{m,\ell}, \qquad m\in\N\cup\{0\}, \, \ell\in\N,
\end{equation*}
to indicate the dependence of \eqref{eq:eigenvaluesTheta} on $\h$, namely via the sequence $b_\ell$ determined in \eqref{eq:ambl} by the Laplace spectrum $\{\lambda_\ell\}$ of $(\Sigma^2,\h)$. Note that $\Theta_{m,0}$, marked in blue in \Cref{fig:ordering-thetas}, are independent of $\h$ because $\lambda_0=0$ so $b_0=\tfrac{1}{2}i$ for all $\h$.
Recall that standard results for Banach manifolds can be used on the set $\mathcal H(\Sigma^2)$ of smooth hyperbolic metrics on $\Sigma^2$ with the Whitney $C^\infty$ topology, see~\cite[Remark~2.1]{BPS}.

\subsection{Morse index and nullity}
Let us compute the Morse index and nullity of $\g$ as a solution to the fractional Yamabe problem, in terms of the eigenvalues~$\Theta_{m,\ell}(\h)$.

\begin{proposition}\label{prop:Morse-idx-nullity}
Given $0<\gamma<\frac{n}{2}-1$, if $\Theta_{m,\ell}(\h) \leq \frac{n+2\gamma}{n-2\gamma} \,\Theta_{0,0}$, then $m=0$ and $\ell\geq 1$. In particular, the Morse index and nullity of $\g=\g_{\S^{n-2}} + \h$ as a solution to the fractional Yamabe problem on $M=\S^{n-2}\times \Sigma^{2}$, i.e., of the critical point $u\equiv 1$ of the functional \eqref{eq:energy} subject to the constraint \eqref{eq:constraint}, are given by:
\begin{equation}\label{eq:Morse-idx-nullity}
\begin{aligned}
\ind(\g)&=\# \left\{\ell\in \N : \Theta_{0,\ell}(\h)<\tfrac{n+2\gamma}{n-2\gamma}\, \Theta_{0,0}\right\},\\
\dim\ker J_\gamma^{\g}&=\# \left\{\ell\in \N : \Theta_{0,\ell}(\h) = \tfrac{n+2\gamma}{n-2\gamma}\, \Theta_{0,0}\right\}.
\end{aligned}
\end{equation}
\end{proposition}

\begin{proof}
By \eqref{am}, \eqref{eq:eigenvaluesTheta} and \eqref{eq:ambl}, as $k=1$, we have $a_1=\frac{n-1}{2}$, $a_0=\frac{n-3}{2}$, and $b_0=\frac{1}{2}i$,~so
\begin{equation}\label{eq:T10andT00}
  \Theta_{1,0}=4^\gamma\, \frac{\Gamma\big(\frac{n+2\gamma}{4} \big)}{\Gamma\big( \frac{n-2\gamma}{4} \big)}\,\frac{\Gamma\big(\frac{n+2+2\gamma}{4} \big)}{\Gamma\big(  \frac{n+2-2\gamma}{4} \big)}, \qquad
  \Theta_{0,0} = 4^{\gamma}\, \frac{\Gamma\big(\frac{n+2\gamma}{4}\big)}{\Gamma\big(\frac{n-2\gamma}{4}\big)} \,\frac{\Gamma\big(\frac{n-2+2\gamma}{4}\big)}{\Gamma\big(\frac{n-2-2\gamma}{4}\big)}.
\end{equation}
Thus, using the identity $\Gamma(z+1)=z\,\Gamma(z)$ with $z=\frac{n-2\pm2\gamma}{4}$, it follows that
\begin{equation}\label{eq:ineq1000}
  \Theta_{1,0}=\tfrac{n+2\gamma-2}{n-2\gamma-2}\,\Theta_{0,0} > \tfrac{n+2\gamma}{n-2\gamma}\,\Theta_{0,0}.
\end{equation}
By \Cref{prop:monotonicity-theta}, we have that $\Theta_{m,\ell}(\h)\geq \Theta_{1,0}$ for all $m\geq 1$ and $\ell\geq0$.
Together with \eqref{eq:ineq1000}, this shows that if $\Theta_{m,\ell}(\h) \leq \frac{n+2\gamma}{n-2\gamma} \,\Theta_{0,0}$, then $m=0$ and $\ell\geq 1$, i.e., $\Theta_{m,\ell}$ is one of the eigenvalues in red on \Cref{fig:ordering-thetas}.

As a consequence of the above, we compute the Morse index and nullity of the Jacobi operator \eqref{eq:jacobi} on the space of $L^2$ functions on $(M,\g)$ with zero average. Namely, by \Cref{prop:eigenvalues-P}, these are respectively given by:
\begin{equation*}
\begin{aligned}
\ind(\g)&=\# \left\{(m,\ell)\in \N\times \N\setminus\{(0,0)\} : \Theta_{m,\ell}(\h)<\tfrac{n+2\gamma}{n-2\gamma}\,Q_{\gamma}^{\g}\right\},\\
\dim\ker J_\gamma^{\g}&=\# \left\{(m,\ell)\in \N\times \N\setminus\{(0,0)\} : \Theta_{m,\ell}(\h)=\tfrac{n+2\gamma}{n-2\gamma}\,Q_{\gamma}^{\g}\right\},
\end{aligned}
\end{equation*}
where $Q_{\gamma}^{\g}=\Theta_{0,0}=Q_\gamma(n,1)$ is given by \eqref{eq:trivial-fractional-curvature}, since $(M,\g)$ is locally isometric to $(\S^{n-2}\times \H^{2},\gp)$. By the above, these quantities agree with \eqref{eq:Morse-idx-nullity}.
\end{proof}

Let us define a convenient positive constant that depends on $n$ and $\gamma<\frac{n}{2}-1$,
\begin{equation}\label{eq:Xi}
 \Xi :=\Theta(a_0,0) = 4^\gamma\,\frac{\Gamma\big(\frac{1+\gamma}{2}+\frac{n-3}{4} \big)^2}{\Gamma\big(\frac{1-\gamma}{2}+\frac{n-3}{4} \big)^2}= 4^\gamma\,\frac{\Gamma\big(\frac{n}{4} + \frac{\gamma}{2}-\frac14 \big)^2}{\Gamma\big( \frac{n}{4} - \frac{\gamma}{2}-\frac14  \big)^2},
\end{equation}
cf.~\eqref{eq:ambl} and \eqref{eq:theta-fct}, recalling that $a_0=\frac{n-3}{2}$.
By \Cref{prop:eigenvalues-P}, $\Xi$ is an eigenvalue of $P_\gamma^{\g}$ if and only if $\lambda=\tfrac14$ is an eigenvalue of the Laplacian on $(\Sigma^2,\h)$.
These special constants play a crucial role in our spectral analysis due to the following facts about Laplace eigenvalues $\lambda_\ell=\lambda_\ell(\Sigma^2,\h)$ of a closed hyperbolic surface $(\Sigma^2,\h)$:
\begin{enumerate}[label = \rm (\alph*)]
  \item If $\ell\geq 2\gen(\Sigma^2)-2$, then $\lambda_\ell(\Sigma,\h)\geq \frac14$ for all $\h\in\mathcal H(\Sigma^2)$;
  \item For all $\varepsilon>0$ and $\ell\in\N$, there exists $\h\in\mathcal H(\Sigma^2)$ such that $\lambda_\ell(\Sigma^2,\h)<\frac14+\varepsilon$.
\end{enumerate}
Fact (a) is the main result in \cite{small-evals}, while (b) is a well-known result of Buser, see e.g.~\cite[Thm~8.1.2]{Bu}.
From (a) and \eqref{eq:ambl}, it follows that if $\ell\geq 2\gen(\Sigma^2)-2$, then
$b_\ell\geq0$, and hence $\Theta_{0,\ell}(\h)\geq \Xi$, for all $\h\in\mathcal H(\Sigma^2)$, by the same arguments using \Cref{lemma:monotonicity} in the proof of \Cref{prop:monotonicity-theta}. On the other hand,  using (b), we can choose $\h$ to pinch $\Theta_{0,\ell}(\h)$ near $\Xi$ for arbitrarily many $\ell$'s, namely:

\begin{lemma}\label{lemma:pinch}
For all $\varepsilon>0$ and $\ell\in \N$, there exists $\h\in \mathcal{H}(\Sigma^2)$ such that
\begin{equation*}
  \Theta_{0,\ell}(\h)<\Xi +\varepsilon.
\end{equation*}
\end{lemma}

\begin{proof}
  This is an immediate consequence of (b) and continuity of $\Theta$, see \eqref{eq:theta-fct}.
\end{proof}

We also have a generic avoidance principle for the eigenvalues $\Theta_{0,\ell}(\h)$, namely:

\begin{lemma}\label{lemnon}
For each $\vartheta >\Xi$, the following set is open and dense in $\mathcal H(\Sigma^2)$:
\begin{equation*}
\mathcal{H}^{P_\gamma}_{\vartheta}(\Sigma^2):=\big\{\h\in \mathcal{H}(\Sigma^2) : \Theta_{0,\ell}(\h)\neq \vartheta, \, \forall \ell\in\N\big\}
\end{equation*}
\end{lemma}

\begin{proof}
By \Cref{lemma:monotonicity} (iii), the function $b\mapsto \Theta(a_0,b)$ is strictly increasing for all $b>0$, and hence $\Theta(a_0,b)>\Theta(a_0,0)=\Xi$ for all $b>0$. Moreover, if $\vartheta>\Xi$, then
\begin{equation*}
  \Theta\!\left(a_0,\sqrt{\lambda-\tfrac14}\right)=\vartheta
\end{equation*}
is satisfied by a unique value $\lambda=\lambda(\vartheta) >\tfrac14$. Thus, since $\Theta_{0,\ell}(\h)=\vartheta$ if and only if $\lambda_\ell(\Sigma^2,\h)=\lambda(\vartheta)$, the desired result follows from \cite[Prop.~2.4]{BPS}.
\end{proof}

Combining the above, we obtain the main ingredient to prove Theorem~\ref{mainthmA}:

\begin{proposition}\label{propimp}
Assume that $0<\gamma<\frac{n}{2}-1$ and that $\Xi$ in \eqref{eq:Xi} satisfies
\begin{equation}\label{ineqimp}
\Xi <\tfrac{n+2\gamma}{n-2\gamma}\, \Theta_{0,0}. 
\end{equation}
Then, given any product metric $\g=\g_{\S^{n-2}}+ \h$ on $M=\S^{n-2}\times\Sigma^2$, with $\h\in \mathcal H(\Sigma^2)$, and any $d\in \N$, there exists a real-analytic path $\h(t)$, $t\in [0,1]$, in $\mathcal H(\Sigma^2)$ such that the corresponding path of product metrics $\g(t):=\g_{\S^{n-2}} + \h(t)$ satisfies:
\begin{enumerate}[label = \rm (\roman*),leftmargin=*,itemsep=2pt]
\item $\g(0)$ is arbitrarily close to $\g$,
\item $\ker J_\gamma^{\g(0)} = \ker J_\gamma^{\g(1)}=\{0\}$,
\item $\ind(\g(1))\geq d+\ind(\g(0)).$
\end{enumerate}
\end{proposition}

\begin{proof}
From \Cref{prop:Morse-idx-nullity} and \Cref{lemnon}, we can choose $\h(0)\in \mathcal{H}(\Sigma^2)$ arbitrarily close to $\h$, such that $\g(0)=\g_{\S^{n-2}}+\h(0)$ is nondegenerate, i.e., $\ker J_\gamma^{\g(0)}=\{0\}$. By \eqref{ineqimp}, there exists $\varepsilon>0$ such that $\Xi+\varepsilon< \tfrac{n+2\gamma}{n-2\gamma}\, \Theta_{0,0}$. Thus, \Cref{prop:Morse-idx-nullity} and \Cref{lemma:pinch,lemnon} imply that there exists $\h(1)\in\mathcal H(\Sigma^2)$ such that $\Theta_{0,\ell}(\h_1)<\Xi+\varepsilon$ for arbitrary $\ell\in\N$ and $\g(1)=\g_{\S^{n-2}}+\h(1)$ is nondegenerate. Choosing $\ell$ sufficiently large,  $\ind(\g(1))-\ind(\g(0))$ can be made arbitrarily large, see \eqref{eq:Morse-idx-nullity}. The existence of a real-analytic path $\h(t)$, $t\in [0,1]$ joining the prescribed endpoints $\h(0)$ and $\h(1)$ in $\mathcal H(\Sigma^2)$ follows from  connectedness of the real-analytic manifold $\mathcal H(\Sigma^2)$.
\end{proof}

In order to use \Cref{propimp}, we provide the following criterion to verify \eqref{ineqimp}.

\begin{proposition}\label{prop:cn}
For each $n\geq 4$, there exists a positive constant $c_{n}<\frac{n}{2}-1$, such that \eqref{ineqimp} holds if and only if $0<\gamma<c_n$.
The sequence $(c_n)_{n\geq4}$ is increasing and asymptotic to $\frac{n}{2}-1$ as $n\nearrow+\infty$. Some approximate values of $c_n$ for small $n$ are:
\begin{equation}
  c_4\approx 0.857, \quad c_5\approx 1.408,\quad c_6\approx 1.932, \quad c_7\approx 2.446, \quad c_8\approx 2.955;
\end{equation}
in particular, $c_n > \frac12$ for all $n\geq4$, $c_n > 1$ for all $n\geq 5$, and $c_n > 2$ for all $n\geq 7$.
\end{proposition}

\begin{proof}
From \eqref{eq:T10andT00} and \eqref{eq:Xi}, we have that \eqref{ineqimp} is equivalent to
\begin{equation*}
  \frac{\Gamma\big(\frac{n}{4} + \frac{\gamma}{2}-\frac14 \big)^2}{\Gamma\big( \frac{n}{4} - \frac{\gamma}{2}-\frac14  \big)^2} < \frac{n+2\gamma}{n-2\gamma} \,\frac{\Gamma\big(\frac{n+2\gamma}{4}\big)}{\Gamma\big(\frac{n-2\gamma}{4}\big)} \,\frac{\Gamma\big(\frac{n-2+2\gamma}{4}\big)}{\Gamma\big(\frac{n-2-2\gamma}{4}\big)}=\frac{\Gamma\big(\frac{n}{4}+\frac{\gamma}{2}+1\big)}{\Gamma\big(\frac{n}{4}-\frac{\gamma}{2}+1\big)}\, \frac{\Gamma\big(\frac{n}{4}+\frac{\gamma}{2}-\frac{1}{2}\big)}{\Gamma\big(\frac{n}{4}-\frac{\gamma}{2}-\frac{1}{2}\big)},
\end{equation*}
where the last step uses that $\Gamma(z+1)=z\,\Gamma(z)$. In turn, the above is equivalent to
\begin{equation}\label{eq:ineq-crossratio}
\frac{\Gamma\big(\frac{n}{4}-\frac{\gamma}{2}+1\big)\,\Gamma\big(\frac{n}{4}-\frac{\gamma}{2}-\frac{1}{2}\big)}{\Gamma\big(\frac{n}{4}-\frac{\gamma}{2}-\frac{1}{4}\big)^2}<\frac{\Gamma\big(\frac{n}{4}+\frac{\gamma}{2}+1\big)\,\Gamma\big(\frac{n}{4}+\frac{\gamma}{2}-\frac{1}{2}\big)}{\Gamma\big(\frac{n}{4}+\frac{\gamma}{2}-\frac{1}{4}\big)^2}.
\end{equation}
Clearly, \eqref{eq:ineq-crossratio} holds if and only if $F\big(\frac{n}{2}-\gamma\big)<F\big(\frac{n}{2}+\gamma\big)$, where $F\colon (1,+\infty)\to\R$ is
\begin{equation}
  F(x) := \frac{\Gamma\big(\tfrac{x}{2}+1\big)\,\Gamma\big(\tfrac{x}{2}-\tfrac12\big)}{\Gamma\big(\tfrac{x}{2}-\tfrac14\big)^2}.
\end{equation}
Routine computations show that $F(x)$ has a global minimum at $x=x_0\approx1.514$, is decreasing if $x<x_0$, and increasing if $x>x_0$.
Moreover, $F(x)$ is asymptotic to:
\begin{equation}\label{eq:F-asympt}
  F(x)\approx  \frac{\sqrt\pi}{\Gamma\big(\tfrac14\big)^2(x-1)}  \text{ as } x\searrow1, \qquad \text{ and} \qquad  F(x)\approx \frac{x}{2} \text{ as } x\nearrow+\infty.
\end{equation}
Thus, given $n\geq4$, as $\frac{n}{2}>x_0$, it follows that $F\big(\frac{n}{2}-\gamma\big)<F\big(\frac{n}{2}+\gamma\big)$ for all $0<\gamma<c_n$, where $c=c_n$ is the unique solution to the equation $F\big(\frac{n}{2}-c\big)=F\big(\frac{n}{2}+c\big)$. The asymptotic behavior of $c_n$ as $n\nearrow+\infty$ is an easy consequence of \eqref{eq:F-asympt}.
\end{proof}

\begin{remark}
  For $n=3$, inequality \eqref{ineqimp} does not hold for any $0<\gamma<\frac12$.
\end{remark}

\subsection{Bifurcation}
Let $\g(t)$, $t\in [0,1]$, be a path of metrics on $M=\S^{n-2}\times \Sigma^2$ with constant fractional curvature and volume given by
\begin{equation*}
\begin{aligned}
  Q_\gamma^{\g(t)} &= Q_\gamma(n,1) =4^{\gamma}\, \frac{\Gamma\big(\frac{n+2\gamma}{4}\big)}{\Gamma\big(\frac{n-2\gamma}{4}\big)} \,\frac{\Gamma\big(\frac{n-2+2\gamma}{4}\big)}{\Gamma\big(\frac{n-2-2\gamma}{4}\big)}, \\[2pt]
   \Vol(M,\g(t))&=\Vol(M,\gp)=\frac{8\pi^{(n+1)/2}}{\Gamma\big(\frac{n-1}{2}\big)}(\gen(\Sigma^2)-1). 
   \end{aligned}
\end{equation*}
 We say that $t=t_*$ is a \emph{bifurcation instant} for $\g(t)$ if there exist sequences $(t_q)_{q\in\N}$ in $[0,1]$ converging to $t_*$ and of metrics $(\g_q)_{q\in\N}$ converging to $\g(t_*)$ such that:
 \begin{enumerate}[label = \rm (\roman*),leftmargin=*,itemsep=2pt]
   \item Each $\g_q$ has constant $\gamma$-fractional curvature;
   \item $\g_q$ is conformal to $\g(t_q)$, but $\g_q\neq \g(t_q)$;
   \item $\Vol(M,\g_q)=\Vol(M,\gp)$.
 \end{enumerate}
In other words, $\g_q=u_q^{\frac{4}{n-2\gamma}} \g(t_q)$ are obtained multiplying $\g(t_q)$ by a nonconstant conformal factor $u_q\colon M\to\R$ that satisfies \eqref{eq:P-eigenvalue-prob} and \eqref{eq:constraint} with respect to the metric $\g(t_q)$. Clearly, bifurcation instants correspond to values of $t$ at which the family $\g(t)$ is not a locally unique solution to the (normalized) fractional Yamabe problem.

Following the same Bifurcation Theory framework used in \cite[Thm 3.4]{BPS}, see also \cite[Appendix A]{LPZ}, we now prove Theorem~\ref{mainthmA} in the Introduction:

\begin{proof}[Proof of Theorem~\ref{mainthmA}]
Fix $0<\gamma<c_n$ and let $\g=\g_{\S^{n-2}}+\h$ be a product metric on $M=\S^{n-2}\times\Sigma^2$, with $\h\in\mathcal H(\Sigma^2)$.
By \Cref{prop:cn}, we may apply \Cref{propimp} with any $d\in\N$ and obtain a real-analytic path of product metrics $\g(t)$, $t\in[0,1]$, on $M$ satisfying properties (i)--(iii) in \Cref{propimp}. 
Consider the corresponding $1$-parameter family of functionals $(I_t)_{t\in[0,1]}$ given by the restriction of the energy functional $E_\gamma^{\g(t)}$ given in \eqref{eq:energy} to the submanifold of $W^{\gamma,2}(M)$ defined by \eqref{eq:constraint}, with respect to the metric $\g(t)$. Recall that $u$ is a critical point of $I_t$ if and only if $u^{\frac{4}{n-2\gamma}}\g(t)$ has constant $\gamma$-fractional curvature and the same volume as $(M,\g(t))$.
Clearly, $\g(t)$ have constant $\gamma$-fractional curvature $Q^{\g(t)}_\gamma=Q_\gamma(n,1)$, so the constant function $1$ is a critical point of $I_t$ for each $t\in [0,1]$. Moreover, the critical point $1$ is nondegenerate for $t=0$ and $t=1$, and its Morse index increases by at least $d$ from $t=0$ to $t=1$, as a consequence of (ii) and (iii) in \Cref{propimp}, respectively.
Finally, $I_t$ satisfy a local Palais--Smale condition as a consequence of Fredholmness of \eqref{eq:jacobi}, so a standard variational bifurcation criterion (see e.g.~\cite[Appendix A]{LPZ}) implies that there exists at least one bifurcation instant $t_*$ for $\g(t)$.

Repeating the procedure above with $\h\in\mathcal H(\Sigma^2)$ selected from a path of hyperbolic metrics, we obtain uncountably many paths $\g(t)$ along which bifurcation occurs. These bifurcating solutions lift via the covering $\S^n\setminus \S^1\to M$ to pairwise nonhomothetic periodic solutions to the singular fractional Yamabe problem on $\S^n\setminus \S^1$, arbitrarily close to the trivial solution and with $\gamma$-fractional curvature arbitrarily close to $Q_\gamma(n,1)$ if they are chosen sufficiently close to $\g(t_*)$.
\end{proof}

\begin{remark}
The above sequences of bifurcating solutions $(\g_q)_{q\in\N}$ converge to $\g(t_*)$ in the $C^r$ topology for any $r\geq 2\gamma$. Let us sketch the proof assuming $\gamma\in(0,1)$, which is the most involved case. Writing $\g_q=u_q^{\frac{4}{n-2\gamma}}\g(t_q)$, as $t\mapsto \g(t)$ is continuous in the $C^s$ topology, $s\geq 2\gamma$, then $Q_\gamma^{\g_q}\to Q_\gamma^{\g(t_*)}$. Moreover, since the sequence $u_q$ converges to $1$ in $W^{\gamma,2}$, a standard Moser iteration argument implies convergence in $L^\infty$, see, e.g., \cite[Thm.~3.4]{MQ}. This convergence can be improved to H\"older $C^\alpha$  for some $\alpha\in(0,1)$, see, e.g.,~\cite{Kassmann}. Finally, bootstrapping the Schauder estimates from \cite[Prop.~2.8]{Silvestre}, one obtains $u_q\to 1$ in $C^r$ for any $r\geq2\gamma$. More details on regularity issues with $\gamma\in(0,1)$ can be found in \cite[Sec.~3.1]{delaunay}.
\end{remark}

\begin{remark}\label{rem:geom-nonuniqA}
It is natural to ask if the nonuniqueness proved above is \emph{geometric}, i.e., whether the different conformal factors solving \eqref{eq:P-eigenvalue-prob} yield nonisometric metrics, see \cite[Rmk~4.1]{BPS} and \cite[Rmk~2.2, Rmk.~3.8]{bp-aif}. In principle, these different conformal factors could be related by a change of variables, that is, there may exist a conformal diffeomorphism that pulls back one metric to the other (hence rendering them isometric). The only readily available information in this regard is that no bifurcating solution on $\S^n\setminus\S^1$ is isometric to the trivial solution, since the group of conformal diffeomorphisms of $(\S^{n-2}\times \H^{2},\gp)$ coincides with its isometry group. However, some of the bifurcating solutions could be isometric among themselves.
\end{remark}

\section{\texorpdfstring{Multiplicity of solutions on $\S^n\setminus \S^k$}{Multiplicity of solutions singular along a subsphere}}\label{section:thmB}

In this section, we combine the Aubin-type existence result (\Cref{prop:aubin-ineq}) with infinite towers of finite-sheeted covering maps, following a technique developed in \cite[Sec.~3]{bp-aif},
see also \cite{bps-imrn,andrade2023general,andrade2023nonuniqueness}, to prove Theorem~\ref{mainthmB}.

Recall that a group $G$ is \emph{profinite} if it is isomorphic to the limit $\varprojlim G_s$ of an inverse system $\{G_s\}_{s\in S}$ of finite groups. The \emph{profinite completion} $\widehat G$ of a group $G$ is a profinite group characterized by the universal property that any group homomorphism $G\to H$, where $H$ is profinite, factors uniquely through a homomorphism $\widehat G\to H$. 
There is a natural homomorphism $G\to \widehat G$, and $G$ is called \emph{residually finite} if this homomorphism is injective.
The relevance of these notions for our geometric application is that a closed manifold $\Sigma$ has coverings of arbitrarily large degree if and only if its fundamental group has infinite profinite completion. Endowing these arbitrarily large coverings with locally isometric Riemannian metrics, one obtains:

\begin{lemma}[Lemma~3.6 in \cite{bp-aif}]\label{lemma:tower}
  Given $V>0$ and a closed Riemannian manifold $(\Sigma,\h)$ whose fundamental group has infinite profinite completion, there exists a finite-sheeted regular covering $\Pi\colon \Sigma'\to \Sigma$ such that $\Vol(\Sigma', \Pi^*\h)>V$.
\end{lemma}

For example, the fundamental group of a closed hyperbolic manifold $(\Sigma,\h)$ is infinite and residually finite by the Selberg--Malcev Lemma (see \cite[Sec.~7.5]{ratcliffe}), hence has infinite profinite completion. Thus, any closed hyperbolic manifold $(\Sigma,\h)$ satisfies the hypotheses of \Cref{lemma:tower} and hence admits finite-sheeted Riemannian coverings by hyperbolic manifolds of arbitrarily large volume.

\begin{lemma}\label{lemma:yamabe}
If a closed Riemannian manifold $(M^n,\g)$, $n\geq3$, is not conformally equivalent to $(\S^n,\g_{\S^n})$, then there exists $\varepsilon>0$ such that $\Lambda_{\gamma}(M,[\g])< \Lambda_{\gamma}(\S^{n},[\g_{\S^{n}}])$ for all $\gamma\in \left(\frac{1}{2}-\varepsilon,\frac{1}{2}+\varepsilon\right)\cup (1-\varepsilon,1+\varepsilon)$.
\end{lemma}

\begin{proof}
From \Cref{def:gamma-Yamabe-const}, we have $\Lambda_{\gamma}(M,[\g])=\inf_{\g_u\in [\g]} \lambda_{1}(P_{\gamma}^{\g_u})\Vol(M,\g_u)^{\frac{2\gamma}{n}}$; in particular, $\Lambda_{\gamma}(\S^{n},[\g_{\S^{n}}])=\lambda_{1}(P_{\gamma}^{\g_{\S^{n}}})\Vol(\S^{n},\g_{\S^{n}})^{\frac{2\gamma}{n}}$. Since $(M^n,\g)$ is not conformally equivalent to $(\S^n,\g_{\S^n})$, the Positive Mass Theorem implies that
\begin{equation}\label{strictin}
\Lambda_{\gamma_{0}}(M,[\g])< \Lambda_{\gamma_{0}}(\S^{n},[\g_{\S^{n}}])
\end{equation}
for $\gamma_{0}=\frac{1}{2}$ and $\gamma_0=1$, see \cite{Escobar,Esc} and \cite{Lee-Parker}, respectively.
Thus, by \Cref{prop:aubin-ineq}, the infimum in \eqref{yam} is achieved for such $\gamma_0$, that is, there exist metrics $\g_u\in [\g]$ such that $\Lambda_{\gamma_{0}}(M,[\g])=\lambda_{1}(P_{\gamma_{0}}^{\g_u})\Vol(M,\g_u)^{\frac{2\gamma_{0}}{n}}$ for $\gamma_{0}=\frac{1}{2}$ and $\gamma_0=1$. As the family $\gamma \mapsto P_{\gamma}^{\g_u}$ is smooth, the map $\gamma \mapsto \lambda_{1}(P_{\gamma}^{\g_u})\Vol(M,\g_u)^{\frac{2\gamma}{n}}$ is continuous in a neighborhood of these $\gamma_{0}$, which finishes the proof.
\end{proof}

\begin{proof}[Proof of Theorem~\ref{mainthmB}]
Recall from \Cref{lemma:conf-equiv} that $\S^{n}\setminus \S^{k}\cong \S^{n-k-1}\times \H^{k+1}$ for all $0< k< n$. Choose a compact quotient $M=\S^{n-k-1}\times \Sigma^{k+1}$ of the latter, and a product metric $\g=\g_{\S^{n-k-1}}+\h$ on $M$, where $\h$ is a hyperbolic metric on $\Sigma^{k+1}$.

Since $(M,\g)$ is not conformal to the round sphere, by \Cref{lemma:yamabe}, there exists $\varepsilon_0>0$ such that $\Lambda_\gamma(M,[\g])<\Lambda_{\gamma}(\S^{n},[\g_{\S^{n}}])$ for $\gamma \in \left(\left(\frac{1}{2}-\varepsilon_0,\frac{1}{2}+\varepsilon_0\right)\cup (1-\varepsilon_0,1)\right)\cap  \left(0,\frac{n}{2}-k\right)$. Here, we further impose $\gamma<1$ so that the hypotheses of \Cref{prop:aubin-ineq} are satisfied, and $\gamma<\frac{n}{2}-k$ so that $Q_\gamma^{\g}=Q_{\gamma}(n,k)>0$, cf.~\eqref{eq:Qgamma-prodmetric}.

Two cases may occur. First, if $\g$ satisfies
\begin{equation}\label{eq:toobig}
  \frac{\int_{M} Q_{\gamma}^{\g}\;\mathrm{vol}_{\g}}{\Vol(M,\g)^{\frac{n-2\gamma}{n}}} = Q_\gamma(n,k) \,\Vol(M,\g)^{\frac{2\gamma}{n}} > \Lambda_\gamma(M,[\g]),
\end{equation}
then \Cref{prop:aubin-ineq} implies that $\g$ is not a minimizing solution to the $\gamma$-fractional Yamabe problem on $(M,[\g])$, but there exists a minimizing solution $\g_{(1)}\in [\g]$, i.e., $\g_{(1)}$ attains the infimum in $\Lambda_\gamma(M,[\g])$. In this case, set $\Sigma_1^{k+1}=\Sigma^{k+1}$ and $M_1=M$.
Second, if $\g$ does not satisfy \eqref{eq:toobig}, then we use \Cref{lemma:tower} to find a finite-sheeted regular covering $\Sigma_1^{k+1}\to\Sigma^{k+1}$ such that lifting $\g$ to $M_1:=\S^{n-k-1}\times \Sigma_1^{k+1}$ via the corresponding product covering $\Pi_1\colon M_1\to M$ yields a metric $\Pi_1^*\g$ with $\gamma$-fractional curvature $Q_\gamma(n,k)$ and
\begin{equation*}
 Q_\gamma(n,k) \,\Vol\big(M_1,\Pi_1^*\g\big)^{\frac{2\gamma}{n}}> \Lambda_\gamma(\S^n,[\g_{\S^n}]).
\end{equation*}
Applying \Cref{lemma:yamabe} to $(M_1,\Pi_1^*\g)$, we obtain $\varepsilon_{1}>0$ such that \Cref{prop:aubin-ineq} yields the existence of a minimizing solution $\g_{(1)}$ to the $\gamma$-fractional Yamabe problem on $M_1$, conformal but not homothetic to $\Pi^*_1\g$, for $\gamma \in \left(\left(\frac{1}{2}-\varepsilon_1,\frac{1}{2}+\varepsilon_1\right)\cup (1-\varepsilon_1,1)\right)\cap \left(0,\frac{n}{2}-k\right)$.
At this point, we have our \emph{first} periodic solution on $\S^n\setminus \S^k$, conformal but not homothetic to the trivial solution, given by the pullback of $\g_{(1)}$ to the universal covering $\widetilde{M_1}=\S^{n-k-1}\times\H^{k+1}\cong \S^n\setminus \S^k$.

Since the fundamental group of $\Sigma_1^{k+1}$ is a finite-index normal subgroup of the fundamental group of $\Sigma^{k+1}$, it also has infinite profinite completion. Thus, we may apply \Cref{lemma:tower} to $\Sigma_1^{k+1}$ and find a finite-sheeted regular covering $\Sigma_2^{k+1}\to\Sigma_1^{k+1}$ such that the product covering $\Pi_2\colon M_2\to M_1$, with $M_2:=\S^{n-k-1}\times \Sigma_2^{k+1}$, satisfies
\begin{equation*}
 Q_\gamma^{\g_{(1)}} \,\Vol\big(M_2,\Pi_2^*\g_{(1)}\big)^{\frac{2\gamma}{n}}> \Lambda_\gamma(\S^n,[\g_{\S^n}]).
\end{equation*}
Applying \Cref{lemma:yamabe} to $(M_2,\Pi_2^*\g_{(1)})$, we find $\varepsilon_2>0$ such that \Cref{prop:aubin-ineq} yields a minimizing solution $\g_{(2)}$ to the $\gamma$-fractional Yamabe problem on $M_2$, conformal but not homothetic to $\Pi^*_2\g_{(1)}$, for $\gamma \in \left(\left(\frac{1}{2}-\varepsilon_2,\frac{1}{2}+\varepsilon_2\right)\cup (1-\varepsilon_2,1)\right)\cap \left(0,\frac{n}{2}-k\right)$. Thus, we obtain a \emph{second} periodic solution on $\S^n\setminus \S^k$ given by the pullback of $\g_{(2)}$ to the universal covering $\widetilde{M_2}=\S^{n-k-1}\times\H^{k+1}$.

Repeating this process $N$ times, we obtain the desired $N$ pairwise nonhomothetic periodic solutions to the singular $\gamma$-fractional Yamabe problem on $\S^n\setminus \S^k$, given by the pullbacks of $\g_{(j)}$, $1\leq j\leq N$, to the universal covering $\widetilde{M_j}=\S^{n-k-1}\times\H^{k+1}$,
for $\gamma \in \left(\left(\frac{1}{2}-\varepsilon,\frac{1}{2}+\varepsilon\right)\cup (1-\varepsilon,1)\right)\cap  \left(0,\frac{n}{2}-k\right)$, where $\varepsilon=\min\{\varepsilon_{0},\varepsilon_1,\varepsilon_2,\dots,\varepsilon_{N}\}$.

Finally, let us discuss the case $k=0$, for which $\S^n\setminus \S^0\cong \R^n\setminus \{0\}\cong \S^{n-1}\times \R$.
The above proof can be repeated \emph{mutatis mutandis}, using the compact quotients $M_L=\S^{n-1}\times \S^1(L)$, where $\S^1(L)=\R/L\Z$, which form a similar tower of finite-sheeted regular coverings with arbitrarily large volume.
Applying \Cref{lemma:yamabe} and \Cref{prop:aubin-ineq}, one produces minimizing solutions at each level that lift to pairwise nonhomothetic periodic solutions
in the universal covering $\widetilde{M_L}=\S^{n-1}\times \R$.
\end{proof}

\begin{remark}
Similar nonuniqueness results for integer $\gamma$ that appear in \cite{bp-aif,bps-imrn,andrade2023nonuniqueness,andrade2023general} yield \emph{infinitely many} solutions, since the iteration process above can be repeated indefinitely, i.e., one may let $N\nearrow+\infty$. However, it is unclear if this is possible in our fractional setting, as there may be no uniform lower bound on $\varepsilon_j>0$ obtained from applying \Cref{lemma:yamabe} at each step $(M_j,\Pi_j^*\g_{(j-1)})$; so we restrict to finitely many steps $1\leq j\leq N$ to ensure that $\varepsilon=\min_{0\leq j\leq N} \varepsilon_j>0$.

Incidentally, the above limitation also means that we cannot apply the clever argument in the proof of \cite[Thm.~1.3]{andrade2023general} to establish \emph{geometric} nonuniqueness of solutions. Indeed, this argument relies on letting $N\nearrow+\infty$ to show that failure of geometric nonuniqueness for infinitely many periodic solutions would imply the existence of a sequence of conformal diffeomorphisms of $\S^{n-k-1}\times\H^{k+1}=\widetilde M_j$ without any convergent subsequence, in contradiction with the Ferrand--Obata Theorem.
As in~\Cref{rem:geom-nonuniqA}, the only readily available information here is that none of the $N$ solutions on $\S^n\setminus \S^k$ produced in the proof of Theorem~\ref{mainthmB} are isometric to the trivial solution, because the group of conformal diffeomorphisms of $(\S^{n-k-1}\times \H^{k+1},\gp)$ coincides with its isometry group.
\end{remark}


\newcommand{\etalchar}[1]{$^{#1}$}
\providecommand{\bysame}{\leavevmode\hbox to3em{\hrulefill}\thinspace}
\providecommand{\MR}{\relax\ifhmode\unskip\space\fi MR }
\providecommand{\MRhref}[2]{%
  \href{http://www.ams.org/mathscinet-getitem?mr=#1}{#2}
}
\providecommand{\href}[2]{#2}

\end{document}